\documentclass[a4paper]{amsart}

\usepackage{color}
\usepackage{amsxtra}
\usepackage{amssymb}
\usepackage{mathtools}
\usepackage{enumerate}
\usepackage{nicefrac}
\mathtoolsset{showonlyrefs} 
\usepackage{ mathrsfs }
\usepackage{tikz}
\usetikzlibrary{positioning, calc}
\usetikzlibrary{shadings}
\usepackage{color}
\usepackage{hyperref}
\hypersetup{
  colorlinks   = true, 
  urlcolor     = blue, 
  linkcolor    = blue, 
  citecolor   = red 
}

\newtheorem{theorem}{Theorem}[section]
\newtheorem{lemma}{Lemma}[section]

\newtheorem{definition}{Definition}[section]
\newtheorem{corollary}{Corollary}[section]
\newtheorem{remark}{Remark}[section]

\title[Eigenvalues for a local and nonlocal
$p-$Laplacian]{
Eigenvalues for a combination between local and nonlocal
$p-$Laplacians
}

\author[L. M. Del Pezzo]{Leandro M. Del Pezzo}
	\address{Leandro M. Del Pezzo \hfill\break\indent
		CONICET and UTDT \hfill\break\indent
		Departamento de Matem\'aticas y
		Estad\'istica
		\hfill\break\indent Universidad Torcuato Di Tella
		\hfill\break\indent Av. Figueroa Alcorta 7350 (C1428BCW)
		\hfill\break\indent Buenos Aires, ARGENTINA. }
	\email{ldelpezzo@utdt.edu}
	\urladdr{http://cms.dm.uba.ar/Members/ldpezzo/}

\author[R. Ferreira]{Ra\'ul Ferreira}
\address{Ra\'ul Ferreira\hfill\break\indent
Departamento de Matem\'atica Aplicada,
\hfill\break\indent Fac. de C.C. Qu\'{\i}micas, U. Complutense de Madrid,
\hfill\break\indent 28040,  Madrid, Spain.
 }
\email{raul\_ferreira@mat.ucm.es}

\author[J. D. Rossi]{Julio D. Rossi}
    \address{Julio D. Rossi\hfill\break\indent
        CONICET and Departamento  de Matem{\'a}tica, FCEyN,\hfill\break\indent
        Universidad de Buenos Aires,
        \hfill\break\indent Pabellon I,
         Ciudad Universitaria (C1428BCW),
         \hfill\break\indent
        Buenos Aires, Argentina.}
\email{jrossi@dm.uba.ar}
\urladdr{http://mate.dm.uba.ar/~jrossi/}

\begin{document}

\begin{abstract}
In this paper we study the Dirichlet eigenvalue problem
$$
\begin{cases}
			-\Delta_p u-\Delta_{J,p}u =\lambda|u|^{p-2}u &\text{ in } \Omega,\\
			u=0 &\text{ in } \Omega^c=\mathbb{R}^N\setminus\Omega.
		\end{cases}
$$
Here $\Delta_p u$ is the standard local $p-$Laplacian, $\Delta_{J,p}u$ is a nonlocal, $p-$homogeneous operator of order zero and $\Omega$ is a bounded domain in $\mathbb{R}^N$.

We show that the first eigenvalue (that is isolated and simple)
satisfies $(\lambda_1)^{1/p}\to \Lambda$ as $p\to\infty$
where $\Lambda$ can be characterized in terms of the geometry of $\Omega$.
We also find  that the eigenfunctions converge, $u_\infty=\lim_{p\to\infty} u_p$, and
find
the limit problem that is satisfied in the limit.
\end{abstract}

\maketitle

\section{Introduction}

Eigenvalue problems for the $p-$Laplacian is by now a classical subject.
In fact, the study of eigenvalues, that is, nontrivial solutions to
$$
- \Delta_p u = - \mbox{div} (|\nabla u|^{p-2} \nabla u) = \lambda |u|^{p-2}u
$$
in a bounded smooth domain, $\Omega$, with Dirichlet boundary conditions
$u=0$ on $\partial \Omega$ starts with \cite{peral} (we refer to \cite{lind} and
references therein on the $p-$Laplacian). For this eigenvalue problem it is known that
there exists a sequence $\lambda_n \to \infty$ of eigenvalues, that the first (the smallest)
eigenvalue is isolated and simple and has a positive associated eigenfunction. The description
of the complete spectrum remains as a mayor open problem. There are two extreme cases that
deserved special attention, $p=1$ (related to the Cheeger problem, see \cite{parini}) and $p=\infty$
(this case is understood as the limit as $p\to \infty$, see \cite{JL,JLM}). The analysis of this
kind of problems goes far beyond this brief description, including eigenvalues with weights, variable
exponents, systems, etc.

On the other hand, recently,
there is an increasing interest in the study of nonlocal operators, like the fractional $p-$Laplacian.
For references concerning nonlocal fractional
	problems we refer to \cite{caffa,Jilka,LL,Hich} and
	references therein.
For the eigenvalue problem for this operator we refer to \cite{Brasco,DPQ,delpe,delpeRossi,Farca,LL,FP} where a detailed study was carry
over showing similarities and differences with the local case (one of the biggest differences is that
the restriction of an eigenfunction to a nodal domain is not an eigenfunction of this nodal domain due to the nonlocal
character of the problem).  Also in \cite{LL} the limit case $p\to \infty$ was studied in the fractional setting.
Concerning nonlocal operators of zero-order, that is, problems of the form
\[
		- \int_{\mathbb{R}^N}J(x-y)|u(x)-u(y)|^{p-2}(u(y)-u(x))\, dy = \lambda |u|^{p-2} u
	\]
involving non-singular and compactly supported kernels (with $u=0$ in $\mathbb{R}^N \setminus \Omega$)
we refer to \cite{ellibro,andreu2} where
it is shown that the constant that corresponds to the infimum of the corresponding Raleigh quotient
is strictly positive. However, it is not known that it is attained due to the lack of compactness.

	The purpose of this paper is to study the following Dirichlet eigenvalue problem
	\begin{equation}\label{eq:EP}
	    \begin{cases}
			\mathcal{L}_{J,p}(u) =\lambda|u|^{p-2}u &\text{ in } \Omega,\\
			u=0 &\text{ in } \Omega^c:=\mathbb{R}^N\setminus\Omega.
		\end{cases}
	\end{equation}
	Here $\Omega$ be a bounded connected domain in $\mathbb{R}^N$ with $C^{1,\alpha}$
	boundary, $p\geq 2$  and
	\[
		\mathcal{L}_{J,p}(u)\coloneqq-\Delta_p u-\Delta_{J,p}u
	\]
	where $$\Delta_p(u)\coloneqq\text{div}(|\nabla u|^{p-2}\nabla u)$$ is the usual local $p-$Laplacian
	operator and
	\[
		\Delta_{J,p}u\coloneqq 2\int_{\mathbb{R}^N}J(x-y)|u(x)-u(y)|^{p-2}(u(y)-u(x))\, dy,
	\]
	where the kernel $J\colon\mathbb{R}^N\to\mathbb{R}$ is a radially symmetric, nonnegative continuous function
	with compact support,  $J(0)>0$  and
	$	\int_{\mathbb{R}^N}J(x) dx=1$,	
	(this last condition is imposed just for normalization).
		Since we assumed $J$ to be continuous and compactly supported
		we let
		\[
			R_J\coloneqq\inf \left\{R\colon J(x)=0 \text{ in } \mathbb{R}^N \setminus B(0,R)\right\}\in(0,+\infty),
		\]
		and we also assume that $J(x) >0$ for $|x|< R_J$.
	
	A scalar $\lambda$ is called an eigenvalue of $\mathcal{L}_{J,p}$
	if there is a nontrivial weak solution $u\in W^{1,p}_0(\Omega)$ of
	\eqref{eq:EP}, that is, a function $u$ such that
	\[
	    \mathcal{H}_{J,p}(u,v)=\lambda\int_\Omega |u|^{p-2}uv dx
	    \quad \forall v\in W^{1,p}_0(\Omega),
	\]
	holds. Here
	\begin{align*}
		\mathcal{H}_{J,p}(u,v)
				\coloneqq&\int_{\Omega}|\nabla u|^{p-2}\nabla u\nabla v dx\\
				& +				\int_{\mathbb{R}^N}\int_{\mathbb{R}^N}J(x-y)|u(x)-u(y)|^{p-2}(u(x)-u(y))
				(v(x)-v(y)) dx dy.	
	\end{align*}
	Such $u$ is is called an eigenfunction corresponding to the eigenvalue $\lambda.$

	From the usual variational techniques we have that the first (the smallest) eigenvalue of \eqref{eq:EP} is
	\[
	    \lambda_1(p)\coloneqq\inf\left\{
	    \dfrac{\displaystyle \mathcal{H}_{J,p}(u,u)}
	    {\displaystyle\|u\|_p^p}\colon
	    u\in W^{1,p}_0(\Omega)\setminus\{0\}
	    \right\},
	\]	
	where $\|\cdot\|_p$ denotes the $L^p(\Omega)-$norm.
	
	Moreover, using a topological tool (the genus), we can construct a sequence of
	eigenvalues $\{\lambda_k\}_{k\in\mathbb{N}}$ of \eqref{eq:EP} such that
	$\lambda_k\nearrow \infty$ as $k\to\infty$.

	\begin{theorem}\label{teo.1.intro}
	   The value $\lambda_1(p)$ is the first eigenvalue of
	   \eqref{eq:EP} (if $\lambda$ is an eigenvalue of \eqref{eq:EP} then
	   $\lambda\ge\lambda_1(p)$). In addition  $\lambda_1(p)$
	   is simple and isolated,
	   and its corresponding eigenfunctions have constant sign.

	   Moreover, there exists a  sequence of
	eigenvalues $\{\lambda_k\}_{k\in\mathbb{N}}$ such that
	$\lambda_k\nearrow \infty$ as $k\to\infty$.
	Every eigenfunction $u$ verifies that $u\in C^{1,\alpha}(\overline{\Omega})$
	for some $\alpha\in(0,1).$
    \end{theorem}

    Now, we deal with the limit as $p\to \infty$.

    \begin{theorem}\label{thm:limiteInfty.intro}
        It holds that
        \[
	       \lim_{p\to \infty} \left[\lambda_1(p)\right]^{\nicefrac1p} =
	        \Lambda
        \]
        where
        \[
            \Lambda\coloneqq\inf\left\{
                \dfrac{\max\left\{\|\nabla u\|_{\infty},
                [u]_J\right\}}{\|u\|_\infty}\colon u\in W^{1,\infty}_0(\Omega)
            \right\}
        \]
        with
        \[
      [u]_J\coloneqq\sup\Big\{|u(x)-u(y)|\colon
            x-y\in \mbox{supp}(J)\Big\}.
    \]

        Moreover, let $u_{p}$ be a nonnegative eigenfunction corresponding to the first eigenvalue $\lambda_1(p)$  and normalize it by
        $\| u_p \|_p=1$. Then,
         there is a sequence $\{p_n\}_{n\in\mathbb{N}}$
        such that $p_n\to \infty,$ $u_{p_n}\to u_\infty \in W^{1,\infty}_0(\Omega)$ uniformly in
        $\overline{\Omega}$ and it holds that
         \[
            \Lambda=\frac{\max\Big\{\|\nabla u_\infty\|_{\infty},
                [u_\infty]_J\Big\}}{\|u_\infty\|_\infty}.
        \]
    \end{theorem}

    This limit value $\Lambda$ can be characterized in terms of the geometry of the
    domain $\Omega$. In fact, in terms of $R_J$ and the inradius
     (the radius of the largest ball contained in $\Omega$), that is
     \[
	    R_\Omega\coloneqq\max \Big\{\mbox{dist}(x,\partial\Omega)\colon x\in\overline{\Omega}\Big\}.
     \]
 	 For the characterization of $\Lambda$ we introduce $K_\Omega\in\mathbb N\cup \{0\}$ and $b\in[0,R_J)$ as
     \[
        R_\Omega=K_\Omega R_J+b,
     \]
that is, we define $K_\Omega$ as the quotient and $b$ as the remainder of the division between $R_\Omega$ and $R_J$.

We distinguish four cases.
\begin{enumerate}
\item If $b=0$ then
	    \[
	        \Lambda=\max\left\{\dfrac1{R_\Omega},\dfrac{1}{K_\Omega}\right\}.
	    \]
\item If $K_\Omega=0$ or $K_\Omega\ne 0$ with $b\ne0$ and $R_J\le 1$ then
	    \[
	        \Lambda=\max\left\{\dfrac1{R_\Omega},\dfrac{1}{K_\Omega+1}\right\}.
	    \]
	    \item If $K_\Omega>0$, $R_J>1$ and $b\ge 1$ then
		 \[
		 	\Lambda=\frac1{K_\Omega+1}.
		 \]
	\item Finally, if $K_\Omega>0$, $1>b>0$ and $R_J>1$ then
		 \[
		 	\Lambda=\dfrac1{K_\Omega+b}.
		 \]
	    \end{enumerate}
	
	    Remark that in this characterization of the limit of the first eigenvalue
	    appears in a nontrivial way the interplay between the two operators
	    involved. The proof is based in a careful choice of test functions
	    for the infimum that defines $\Lambda$. In these test function one can
	    see the local/nonlocal character of the limit problem.

Concerning the equation verified by the limit we introduce the following operators
$$
\Delta_\infty u= \sum_{i,j=1}^N u_{x_i} u_{x_j} u_{x_i x_j} = \langle D^2 u \nabla u, \nabla u \rangle ,
$$
$$
\mathcal L_{J,\infty} u =\sup_{x-y\in supp(J)} \Big(u(y)-u(x)\Big)+
\inf_{x-y\in supp(J)} \Big(u(y)-u(x)\Big).
$$
Following ideas in \cite{andreu}, this operator can be decomposed as follows
$$
\mathcal L_{J,\infty} u=\mathcal L_{J,\infty}^+ u+\mathcal L_{J,\infty}^- u
$$
where
$$
\mathcal L_{J,\infty}^+ u=\sup_{x-y\in supp(J)} \Big(u(y)-u(x)\Big)
\quad \mbox{and}\quad
\mathcal L_{J,\infty}^- u=\inf_{x-y\in supp(J)} \Big(u(y)-u(x)\Big).
$$

\begin{theorem} \label{teo.viscosa.intro} Any uniform limit of eigenfunctions $u_p$
corresponding to the first eigenvalue $\lambda_1(p)$,
    is a viscosity solution to
    \begin{equation}
      \label{eq:limite.intro}
            \begin{cases}
	            \max\Big\{M_1(u), M_2(u)\Big\}=0 &\text{ in }\Omega,\\
	            u =0 &\text{ on }\partial\Omega,
            \end{cases}
    \end{equation}
    where
    $$
    	    M_1(u(x))\coloneqq\min \Big\{-\Lambda u(x)-\mathcal{L}^-_{J,\infty}u(x),
	        -\mathcal{L}_{J,\infty}u(x), -\mathcal L_{\infty,J}^- u(x)-|\nabla u(x)|
	        \Big\}
	        $$
	        and
	           \begin{align*}
            & M_2(u(x))\coloneqq
             \min\Big \{|\nabla u(x)|-\Lambda u(x),
                -\Delta_\infty u(x),
                |\nabla u(x)|-\mathcal L_{\infty,J}^+u(x),\\
                &\qquad\qquad \qquad \qquad\qquad |\nabla u (x)|+\mathcal L_{\infty,J}^-u(x)
                \Big\}.
     \end{align*}
\end{theorem}

Notice that the equation $\max\{M_1(u), M_2(u)\}=0 $ is a sort of eigenvalue problem
(the value $\Lambda$ appears there and if $u$ is a viscosity solution then $k\, u$
is also a solution for every $k>0$). In addition, we highlight the local/nonlocal character
of $M_1(u)$ and $M_2(u)$.

    \begin{remark} {\rm Here we assumed that $p\geq 2$. Some of our results
    can be extended to $1<p<2$ but in this case the equation is singular
    (one of the main difficulties that we found is to establish the validity
    of the strong maximum principle for $1<p<2$, see Lemma \ref{lemma:smp}
    for $p\geq 2$). Also, notice that for $1<p<2$ we need to
			be more careful in the definition of
			viscosity solutions, see
			\cite{JJ} for the $p-$Laplace equation.  }				
		\end{remark}
		
		\begin{remark} {\rm
		Our results can be extended to cover the problems
		\begin{equation}\label{eq:EP.999}
	    \begin{cases}
			-\Delta^s_p u-\Delta_{J,p}u =\lambda|u|^{p-2}u &\text{ in } \Omega,\\
			u=0 &\text{ in } \Omega^c.
		\end{cases}
	\end{equation}
	and
	\begin{equation}\label{eq:EP.000}
	    \begin{cases}
			-\Delta_p u-\Delta^s_{p}u =\lambda|u|^{p-2}u &\text{ in } \Omega,\\
			u=0 &\text{ in } \Omega^c.
		\end{cases}
	\end{equation}
	
	In fact, we have
        \[
	        \left[\lambda_1(p)\right]^{\nicefrac1p}\to
	        \Lambda
        \]
        where
        \[
            \Lambda\coloneqq
            \left\{
            \begin{array}{ll}
            \inf\left\{
                \dfrac{\max\left\{[ u]_{s},
                [u]_J\right\}}{\|u\|_\infty}\colon u\in W^{1,\infty}_0(\Omega)
            \right\} \qquad & \mbox{ for } \eqref{eq:EP.999}, \\[10pt]
            \inf\left\{
                \dfrac{\max\left\{\|\nabla u\|_{\infty},
                [u]_s\right\}}{\|u\|_\infty}\colon u\in W^{1,\infty}_0(\Omega)
            \right\} \qquad & \mbox{ for } \eqref{eq:EP.000}.
            \end{array}
            \right.
        \]
       Here, $[u]_s$ is the $s-$H\"older seminorm,
        \[
      [u]_s\coloneqq\sup_{x,y \in \mathbb{R}^N }\left\{\dfrac{|u(x)-u(y)|}{|x-y|^s}\right\}.
    \]

    We present our results for \eqref{eq:EP} since it involves a combination of operators in the extreme cases, a local
    operator and a zero-order one.
		}
		\end{remark}
	
	\begin{remark} {\rm
		If we consider (for two parameters $\alpha, \beta >0$)
		\begin{equation}\label{eq:EP.iii}
	    \begin{cases}
			- \alpha \Delta_p u- \beta \Delta_{J,p}u =\lambda|u|^{p-2}u &\text{ in } \Omega,\\
			u=0 &\text{ in } \Omega^c,
		\end{cases}
	\end{equation}
	one can easily show that the first eigenvalue for this problem, $\lambda_1(p, \alpha, \beta)$
	verifies similar properties as the ones described here for the case $\alpha=\beta =1$.
	Moreover, it holds that
        \[
	   \lim_{\alpha \to 1, \beta \to 0}  \lambda_1(p, \alpha, \beta)= \lambda_1
	   \qquad \mbox{and} \qquad  \lim_{\alpha \to 0, \beta \to 1}
	   \lambda_1(p, \alpha, \beta)= \lambda_{1,J}
        \]
        where $\lambda_1$ is the first eigenvalue for the $p-$Laplacian and
        $\lambda_{1,J}$ is the infimum of the corresponding Raleigh quotient
        for the nonlocal zero order operator.

        Concerning eigenfunctions we have
         \[
	   \lim_{\alpha \to 1, \beta \to 0}  u_1(p, \alpha, \beta)= u_1
        \]
        with $u_1 $ an eigenfunction of the $p-$Laplacian associated to $\lambda_1$. However,
        one can not obtain the existence of the limit $\lim_{\alpha \to 0, \beta \to 1}  u_1(p, \alpha, \beta)$
        using compactness results since in this case there is no uniform bound for $\| \nabla u\|_p$.

   Therefore, one can look at \eqref{eq:EP.iii} as an eigenvalue problem that
   interpolates between a local
    operator and a zero-order one.
		}
		\end{remark}

\section{Some general results}\label{generalresults}%
	
	\subsection{Boundary regularity}
		Let $\Omega$ be a bounded domain in $\mathbb{R}^N$ with $C^{1,\alpha}$ boundary,
		$K,$ $M_0$ be a positive constants and
		$a\colon\overline{\Omega}\times\mathbb{R}\times\mathbb{R}^N\to\mathbb{R}$
		be such that
		\[
			|a(x,z,l)|\le K(1+|l|)^{p+2}
		\]
		for all $(x,z,l)$ in $\partial\Omega\times[-M_0,M_0]\times\mathbb{R}^N.$
		Before setting our regularity result, we introduce our definition of weak solution to
		\begin{equation}\label{eq:reg}
			\begin{cases}
				\mathcal{L}_{J,p}(u)=a(x,u,\nabla u) & \text{ in } \Omega,\\
				u=0&\text{ in } \Omega^c.
			\end{cases}
		\end{equation}
		\begin{definition}
			A weak solution of \eqref{eq:reg} is a function  $u\in W^{1,p}_0(\Omega)$
			such that
			\[
				\mathcal{H}_{J,p}(u,v)=\int_{\Omega}a(x,u,\nabla u)v \, dx\qquad
				\forall v\in W^{1,p}_0(\Omega),
			\]
			where
			\begin{align*}
				\mathcal{H}_{J,p}(u,v)\coloneqq&\int_{\Omega}|\nabla u|^{p-2}u\nabla v dx\\
				&+\int_{\mathbb{R}^N}\int_{\mathbb{R}^N}J(x-y)|u(x)-u(y)|^{p-2}(u(x)-u(y))
				(v(x)-v(y)) dx dy	.	
			\end{align*}
		\end{definition}

		Observe that if $u\in L^\infty(\mathbb{R}^N),$ then
		$\Delta_{J,p}u\in L^\infty(\mathbb{R}^N).$
		Therefore, as a consequence of \cite[Theorem 1]{regularity} we have our
		regularity result.
		
		\begin{theorem}\label{thm:regularity}
			If $u$ is a bounded weak solution of \eqref{eq:reg} with $\|u\|_{L^\infty(\Omega)}\le M_0$
			then there is a positive constant $\alpha$ such that
			$u\in C^{1,\alpha}(\overline{\Omega}).$	
		\end{theorem}

	\subsection{Comparison principle}
		Given $u,v\in W^{1,p}(\mathbb{R}^N),$ we say that
		\[
			\mathcal{L}_{J,p} u\le\mathcal{L}_{J,p}v \quad \text{in } \Omega,
		\] 	
		in the weak sense, if
		\[
			\mathcal{H}_{J,p}(u,w)\le\mathcal{H}_{J,p}(v,w) \quad\forall
			w\in W^{1,p}_0(\Omega), w\ge0.
		\]
		
		\begin{lemma}[Comparison principle]\label{lem:cp}
			If $u,v\in W^{1,p}(\mathbb{R}^N)$ and
	  		\begin{equation}\label{eq:cp}
				\mathcal{L}_{J,p} u\le\mathcal{L}_{J,p}v \quad \text{in } \Omega,
			\end{equation}
			in the weak sense, and $u\le v$ in $\Omega^c$ then $u\le v$ in $\Omega.$
		\end{lemma}
		\begin{proof}
			By \eqref{eq:cp}, for all $w\in W^{1,p}_0(\Omega),$ $w\ge0$ we have
			that
			\begin{align*}
				0\le& \int_\Omega\left(|\nabla v|^{p-2}\nabla v-
				|\nabla u|^{p-2}\nabla u\right)
				\nabla wdx\\&+\int_{\mathbb{R}^N}\int_{\mathbb{R}^N}J(x-y)
				\mathcal{A}(v,u,x,y)
				\left(w(x)-w(y)\right) dx dy
			\end{align*}
			where
			\[
				\mathcal{A}(v,u,x,y)=|v(x)-v(y)|^{p-2}\left(v(x)-v(y)\right)-
				|u(x)-u(y)|^{p-2}
				\left(u(x)-u(y)\right).
			\]
			Now, we choice $w=\max\{u-v,0\}$ that is nonnegative
			and belongs to $W^{1,p}_0(\Omega)$ by hypothesis.
			Arguing as \cite[Proof of Lemma 9]{LL}, we have
			\[
				\int_{\mathbb{R}^N}\int_{\mathbb{R}^N}J(x-y) \mathcal{A}(v,u,x,y)
				\left(w(x)-w(y)\right) dx dy\le0.
			\]
		
			 On the other hand, by Cauchy-Schwarz and Young inequalities, we
			have
			\[
				 \int_\Omega\left(|\nabla v|^{p-2}\nabla v-|\nabla u|^{p-2}\nabla u\right)
				\nabla wdx\le0.
			\]
			Therefore
			\begin{align*}
				\int_\Omega\left(|\nabla v|^{p-2}\nabla v-|\nabla u|^{p-2}\nabla u\right)
				\nabla w dx&=0,\\
				\int_{\mathbb{R}^N}\int_{\mathbb{R}^N}J(x-y) \mathcal{A}(v,u,x,y)
				\left(w(x)-w(y)\right) dx dy &=0.
			\end{align*}
			Then by \cite[See (2.2)]{simon}, we have that $w\equiv 0$ and therefore
			$v\ge u$ in $\mathbb{R}^N.$
		\end{proof}
		
		\begin{corollary}[Weak maximum principle]
			If $u\in W^{1,p}(\mathbb{R}^N)$ and
	  		\[
				\mathcal{L}_{J,p} u\ge 0 \quad \text{in } \Omega,
			\]
			in the weak sense and $u\ge 0$ in $\Omega^c$ then $u\ge0$ in $\Omega.$
		\end{corollary}

	\subsection{Viscosity solutions}
		Now we introduce our definition of viscosity solution of
		\begin{equation}\label{eq:vs}
			\begin{cases}
				\mathcal{L}_{J,p}u=f(x) &\text{in }\Omega,\\
				u=0&\text{in }\Omega^c.%
			\end{cases}
		\end{equation}
		Here $f\colon\Omega\to\mathbb{R}$ is a continuous function.

		\begin{definition}
			Let $f\colon\Omega\to\mathbb{R}$ be a continuous function and
			$p\ge2.$ A upper semicontinuous function $u\colon\mathbb{R}^N
			 \to [-\infty,+\infty)$
			is a viscosity subsolution of
			\eqref{eq:vs} if
			\begin{enumerate}[(i)]
				\item $u$ is not identically $-\infty.$
				\item For any open set $U\subset\Omega,$ any $x_0\in U$ and any
					$\phi\in C^2(U)$ such that $\phi(x_0)=u(x_0)$ and $\phi\ge u$ in $U$
					if we let
					\[
						v(x)=
							\begin{cases}
								\phi(x) & \text{if } x\in U,\\
								u(x) & \text{if } x\in U^c,
							\end{cases}
					\]
					we have that $$\mathcal{L}_{J,p}v(x_0)\le f(x_0).$$
				\item $u\le0$ in $\Omega^c.$
			\end{enumerate}

			Similarly, a lower semicontinuous function $u\colon\mathbb{R}^N
			\to(-\infty,+\infty]$
			is a viscosity supersolution of \eqref{eq:vs} if
			\begin{enumerate}[(i)]
				\item $u$ is not identically $+\infty.$
				\item For any open set $U\subset\Omega,$ any $x_0\in U$ and any
					$\psi\in C^2(U)$ such that $\psi(x_0)=u(x_0)$ and $\psi\le u$ in $U$
					if we let
					\[
						w(x)=
							\begin{cases}
								\psi(x) & \text{if } x\in U,\\
								u(x) & \text{if } x\in U^c,
							\end{cases}
					\]
					we have that $$\mathcal{L}_{J,p}w(x_0)\ge f(x_0).$$
				\item $u\ge0$ in $\Omega^c.$
			\end{enumerate}

			Finally, a continuous function $u$ is a viscosity solution of \eqref{eq:vs} if
			it is both
			a viscosity subsolution and a viscosity supersolution.
		\end{definition}

		Observe that if $u$ is a bounded weak solution of \eqref{eq:vs} with
		$f\in L^\infty(\Omega),$ from our previous regularity result, Theorem \ref{thm:regularity},
		we have that $u$ is continuous in $\mathbb{R}^N.$
		Then it makes sense to ask if a weak solution is also a viscosity solution. In fact,
		following \cite{JLM} (see also \cite[Section 4]{LL} for a nonlocal counterpart),
		we can show the following result.
		
		\begin{theorem}\label{thm:wkvs}
			Let $f\in C(\overline{\Omega}).$
			 If $u\in W^{1,p}(\Omega)\cap L^{\infty}(\Omega)$
			is a weak solution of \eqref{eq:vs} then $u$ is a viscosity solution  of
			\eqref{eq:vs} .
		\end{theorem}
		
		\begin{proof} Once we know that $u$ is continuous and write the problem as
		$$
		\begin{cases}
			-\Delta^s_p u =\Delta_{J,p}u +\lambda|u|^{p-2}u &\text{ in } \Omega,\\
			u=0 &\text{ in } \Omega^c,
		\end{cases}
		$$
		the result follows exactly as in \cite{JLM} (remark that the right hand side is continuous
		as a function of $x\in \Omega$).
		\end{proof}
		
		To conclude, we set a strong maximum principle for viscosity supersolution.
		
		\begin{lemma}[Strong maximum principle]\label{lemma:smp}
			Let $p\ge2.$ If $u$ is a viscosity supersolution of
			\begin{equation}\label{eq:smpvs}
				\begin{cases}
					\mathcal{L}_{J,p} u=0 & \text{in }\Omega,\\
					u=0 &\text{in }\Omega^c,
				\end{cases}				
			\end{equation}
				then $u=0$ in $\mathbb{R}^N$  or $u>0$ in  $\Omega.$
		\end{lemma}
		\begin{proof}
			 Without loss of generality we can assume $u\ge0$ in $\mathbb{R^N}.$
			 Suppose that there is $x_0\in\Omega$ such that $u(x_0)=0.$
			 Given $\varepsilon >0$ such that $\varepsilon<R_J$ and  $B(x_0,\varepsilon)\subset\Omega,$ we take
			\[
					w(x)=
						\begin{cases}
								\psi(x) & \text{if } x\in B(x_0,\varepsilon) ,\\
								u(x) & \text{if } x\in B(x_0,\varepsilon)^c.
						\end{cases}
			\]
			 Observe that $w(x)\le u(x)$ for all $x\in\mathbb{R}^N.$
			 Then, using that $u$ is a viscosity
			 supersolution of \eqref{eq:smpvs} we have that
			 \[
			 	0\le \mathcal{L}_{J,p}w(x_0)\le -\int_{\mathbb{R}^N}J(x_0-y)w(y)^{p-1} dy.
			 \]
			Since $J$ is nonnegative, we get
			$w(x)=0$ for all $x\in\{x_0\}-\text{supp}(J).$ Therefore
			$u(x)=0$ in $B(x_0,R_J)\setminus B(x_0,\varepsilon).$ Moreover, as $\varepsilon$
			is arbitrary small, we get $u(x)=0$ in $B(x_0,R_J).$
			
			If we repeat this procedure, since $\Omega$ is a bounded connected
			domain, we can obtain an open cover $\{U_n\}_{n=1}^m$ of $\Omega$ such
			that $u=0$ in $U_n$ for all $n\in\{1,\dots,m\},$ that is $u=0$ in $\Omega.$
		\end{proof}

\section{Eigenvalues}\label{eigenvalues}
    We start this section by showing that any eigenfunction is bounded.

    \begin{lemma}\label{lemma:be}
          If $u$ is an eigenfunction corresponding to the eigenvalue $\lambda$
          then $u\in L^\infty(\Omega).$
    \end{lemma}
    \begin{proof}
	    If $p>N,$ by the Sobolev embedding theorem we have that $u\in L^\infty(\Omega).$
	    To extend the result to the case $1<p\le N,$ we will follow ideas from \cite{FP}.
	
	    Observe that it is enough to prove that $u_+\coloneqq
	    \max\{u,0\}\in L^{\infty}(\Omega)$
	    since $-u$ is also an eigenfunction corresponding to $\lambda.$ In fact,
	    since \eqref{eq:EP} is homogeneous, it is enough to prove that
	    \[
	       \|u_+\|_\infty\coloneqq \|u_+\|_{L^{\infty}(\Omega)}\le 1\quad\text{ if }
	        \|u_+\|_{p}\le \delta
	    \]
	    where $\delta>0$ we will determined.
	
	    For all $n\in\mathbb{N}$ we define
	    \[
	        w_n(x)\coloneqq (u-(1-2^{-n}))_+.
	    \]
	
	    Observe that $w_n\in W^{1,p}_0(\Omega).$  Then using that
	        \begin{itemize}
	            \item $|\nabla w_n(x)|^p=|\nabla u(x)|^{p-2}\nabla u(x)\nabla w_n(x)$
	            a.e. in $\mathbb{R}^N,$
	            \item For any $v\in W^{1,p}_0(\Omega),$
	             \[
                |v(x)-v(y)|^{p-2}(v(x)-v(y))(v_+(x)-v_+(y))\ge|v_+(x)-v_+(y)|^p\ge0,
                \]
	            \item $u$ is an eigenfunction corresponding to $\lambda,$
            \end{itemize}
             we have that
            \[
                \int_{\Omega} |\nabla w_{n+1}|^p dx\le
                \mathcal{H}_{J,p}(u,w_{n+1})=
                \lambda\int_{\Omega}|u|^{p-2}uw_{n+1} dx.
            \]
	        The rest of the proof is entirely similar to that of \cite[Theorem 3.2]{FP}.
    \end{proof}

     Thus, by Lemma \ref{lemma:be} and Theorem \ref{thm:regularity} we get

     \begin{corollary}\label{cor:regeigen}
	       If $u$ is an eigenfunction corresponding to the eigenvalue $\lambda$
          then $u\in C^{1,\alpha}(\overline{\Omega})$ for some $\alpha\in(0,1).$
    \end{corollary}

    Moreover, from Corollary \ref{cor:regeigen} and Theorem \ref{thm:wkvs} we obtain
    \begin{corollary}\label{cor:eigenvisc}
	       If $p\ge 2$ and $u$ is an eigenfunction corresponding to the eigenvalue $\lambda$
          then $u$ is a viscosity solution of \eqref{eq:EP}.
    \end{corollary}

    \subsection{The first eigenvalue.}
                First we show that if $u$ is an eigenfunction corresponding to
        $\lambda_1(p)$ then $u$ has constant sign.

        \begin{lemma}\label{lemma:constantsign}
            Let $p\ge 2.$ If  $u$ is an eigenfunction corresponding to
            $\lambda_1(p)$ then $u$ has constant sign.
        \end{lemma}
        \begin{proof}
             Since $u$ is an eigenfunction corresponding to
             $\lambda_1(p),$ so is $|u|.$ Therefore
             we can assume with no loss of generality that $u$ is
             nonnegative. Then, by Corollary \ref{cor:eigenvisc} and Lemma
             \ref{lemma:smp}, we have that $u>0$ in $\Omega.$
        \end{proof}

        Our next result shows that the first eigenvalue is simple.

        \begin{theorem}\label{thm:simple}
            Let $p\ge2$ and fix $u$ a positive
            eigenfunction corresponding to $\lambda_1(p).$
            If $\lambda>0$
            is such that there is a non-negative eigenfunction
            $v$ corresponding to $\lambda$ then $\lambda=\lambda_1(p)$
            and there is $k\in\mathbb{R}$ such that $v=ku$ in $\mathbb{R}^N.$
        \end{theorem}

        \begin{proof}
	        By Corollary \ref{cor:eigenvisc} and Lemma
             \ref{lemma:smp}, we have that $v>0$ in $\Omega.$ Moreover for all
             $n\in\mathbb{N}$
             we have that
             \[
                w_n\coloneqq \dfrac{u^p}{\left(v+1/n\right)^{p-1}}\in W^{1,p}_0(\Omega).
             \]

             Then, by Picone's identity (see \cite{AH}), we have
             \begin{align*}
	                0&\le\int_{\Omega} |\nabla u|^p + (p-1)
	                \dfrac{u^p}{\left(v+1/n\right)^{p}}|\nabla v|^p
	                -p\dfrac{u^{p-1}}{\left(v+1/n\right)^{p-1}}
	                |\nabla v|^{p-2}\nabla v\nabla u\, dx\\
	                &=\int \left(|\nabla u|^p-
	                |\nabla v|^{p-2}\nabla v\nabla w_n\right)dx\\
	                &=\lambda_1(p)\int_{\Omega}|u|^p dx
	                -\lambda\int_{\Omega} |v|^{p-2}v w_n dx\\
	                &\qquad -\int_{\mathbb{R}^N}
	                \int_{\mathbb{R}^N}J(x-y)|u(x)-u(y)|^pdxdy\\
	                & \qquad +  \int_{\mathbb{R}^N}
	                \int_{\mathbb{R}^N}J(x-y)|v(x)-v(y)|^{p-2}
	                (v(x)-v(y))(w_n(x)-w_n(y))dxdy.
             \end{align*}
             On the other hand, by \cite[Lemma 6.2]{A}, we have that
             \begin{align*}
                 &-\int_{\mathbb{R}^N}
	                \int_{\mathbb{R}^N}J(x-y)|u(x)-u(y)|^pdxdy\\
	             &+  \int_{\mathbb{R}^N}
	                \int_{\mathbb{R}^N}J(x-y)|v(x)-v(y)|^{p-2}
	                (v(x)-v(y))(w_n(x)-w_n(y))dxdy\le 0. 	
            \end{align*}
            Therefore
            \begin{align*}
	                0&\le\int_{\Omega} |\nabla u|^p + (p-1)
	                \dfrac{u^p}{\left(v+1/n\right)^{p}}|\nabla v|^p
	                -p\dfrac{u^{p-1}}{\left(v+1/n\right)^{p-1}}
	                |\nabla v|^{p-2}\nabla v\nabla u\, dx\\
	                &\le \lambda_1(p)\int_{\Omega}|u|^p dx
	                -\lambda\int_{\Omega} |v|^{p-2}v w_n dx.
            \end{align*}

            By the dominated convergence theorem and Fatou's lemma we get
            \begin{align*}
	                0&\le \int_{\Omega}\left(|\nabla u|^p + (p-1)
	                \dfrac{u^p}{v^{p}}|\nabla v|^p
	                -p\dfrac{u^{p-1}}{v^{p-1}}
	                |\nabla v|^{p-2}\nabla v\nabla u\right) dx\\
	                &\le (\lambda_1(p)-\lambda)\int_{\Omega}|u|^p dx.
            \end{align*}
            Therefore, $\lambda=\lambda_1(p).$

            Moreover, using again Picone's identity, we obtain
            \[
                |\nabla u|^p + (p-1)
	                    \dfrac{u^p}{v^{p}}|\nabla v|^p
	                -p\dfrac{u^{p-1}}{v^{p-1}}
	                |\nabla v|^{p-2}\nabla v\nabla u=0
	                \quad \text{ a. e. in }  \mathbb{R}^N.
            \]
            Thus, there is $k>0$ with $v=ku$ in $\Omega.$
       \end{proof}

       Our next goal is to show that $\lambda_1(p)$ is isolated.
       For this we need the following technical lemma.

       \begin{lemma}\label{lemma:tecnico}
            Let $p\ge2.$
            If $u$ is an eigenfunction corresponding to $\lambda>\lambda_1(p)$
            then there is a positive constant $C$ independent of $u$ such that
            \[
	            C<|\{u<0\}|.
            \]
        \end{lemma}

        \begin{proof}
	        By Theorem \ref{thm:simple}, we have that $u_{-}\coloneqq\min\{u,0\}\not\equiv0.$
	        Since $u_-\in W^{1,p}_0(\Omega),$ by the  Sobolev embedding theorem
	        for any $q\in (p,p_*)$ there is a constant $C$ independent of $u$ such that
	        \begin{equation}\label{eq:set}
	                        C\|u_-\|_{q}\le\|\nabla u_-\|_{p}.
            \end{equation}
	        Here $p_*$ denotes the Sobolev critical exponent, tha is
	        \[
	            p_*=\begin{cases}
	                \dfrac{Np}{N-p} &\text{ if } p<N\\
	                \infty &\text{ if } p\ge N.
                \end{cases}
	        \]
	
	        Thus, by \eqref{eq:set} we get
	       $$
	            C\|u_-\|_{q}^p\le\|\nabla u_-\|_{p}^p
	            \le \mathcal{H}_{J,p}(u,u_-)
	            \le \lambda \|u_-\|_{L^p(\Omega)}^p.
            $$
            Finally, by H\"older's inequality we get
            \[
                 C\|u_-\|_{p}^p
                 \le \lambda\|u_-\|_{q}^p
                 |\{u<0\}|^{\frac{q-p}{q}}.
            \]
            Therefore
            \[
                \left(\dfrac{C}\lambda\right)^{\frac{q}{p-q}}
                 \le
                 |\{u<0\}|.
            \]
        \end{proof}

        Now, proceeding as in the proof of \cite[Theorem 4.11]{DPQ} we show that
        the next result holds.

        \begin{theorem}\label{thm:isolated}
            Let $p\ge 2.$ Then $\lambda_1(p)$ is isolated.
        \end{theorem}

        \begin{proof} By definition, $\lambda_1(p)$
		is left-isolated. To prove that $\lambda_1(p)$
		is right-isolated, we argue by contradiction. We assume
		that there is a sequence of eigenvalues
		$\{\lambda_k\}_{k\in \mathbb{N}}$ such that
		$\lambda_k \searrow \lambda_1(p)$ as $k\to \infty.$
		Let $u_k$ be an eigenfunction associated to $\lambda_k$ such
		that $\|u_k\|_{p}=1.$ Then $\{u_k\}_{k\in \mathbb{N}}$ is
		bounded in $W_0^{s,p}(\Omega)$ and therefore
		we can extract a subsequence (that we still denoted by
		$\{u_k\}_{k\in \mathbb{N}}$) such that
		\[
				u_k\rightharpoonup u  \text{ weakly in }
				W_0^{s,p}(\Omega),\qquad
				u_k\to u  \text{ strongly in }
				L^p (\Omega).
		\]	
		Then $\|u\|_{p}=1$ and $u$ is an eigenfunction associated to $\lambda_1(p).$
		Therefore $u$ has constant sign.
		
		Now,
		we can arrive to a contradiction. By Egoroff's theorem we can find a subset $A_\delta$
		of $\Omega$ such that $|A_\delta| < \delta$ and $u_k \to u$ uniformly in $\Omega \setminus A_\delta$. From Lemma \ref{lemma:tecnico} and the uniform convergence in  $\Omega \setminus A_\delta$ we obtain that $|\{u>0\}| >0$ and $|\{u>0\}| <0$.
		 This contradicts the fact that an eigenfunction associated with the first eigenvalue does not change sign.
  \end{proof}

        We also have the existence of higher eigenvalues.

        \begin{theorem} \label{teo2.lll}
			There is a sequence of eigenvalues $\lambda_n$ such that
			$\lambda_n\to\infty$ as $n\to\infty$.
		\end{theorem}

		\begin{proof}
			It follows as in \cite{GAP} and hence we omit the details and only sketch the proof
			for the reader's convenience.
			Let us consider
			\[
				M_\alpha = \Big\{u \in W_0^{1,p}(
				\Omega)\colon  \mathcal{H}_{J,p}(u,u) = p \alpha \Big\}
			\]
			and
			\[
				\varphi (u) = \frac{1}{p}
					\int_{\Omega} |u|^p.
			\]
			We are looking for critical points
			of $\varphi$ restricted to the manifold $M_\alpha$ using a minimax
			technique.
			We consider the class
			\[
  				\Sigma =\Big\{A\subset W_0^{1,p}(
				\Omega)\setminus\{0\}
				\colon A \mbox{ is closed, } A=-A \Big\}.
  			\]
			Over this class we define the genus,
			$\gamma\colon\Sigma\to {\mathbb{N}}\cup\{\infty\}$, as
			\[
				\gamma(A) = \min \Big\{k\in {\mathbb{N}}\colon
				\mbox{there exists } \phi\in C(A,{{\mathbb{R}}}^k-\{0\}),
				\ \phi(x)=-\phi(-x)\Big\}.
			\]
			Now, we let $$C_k = \Big\{ C \subset M_\alpha \colon C
			\mbox{ is compact, symmetric and } \gamma ( C) \le k \Big\} $$
			and let
			\begin{equation}
					\label{betak} \beta_k
					= \sup_{C \in C_k} \min_{u \in C} \varphi(u).
			\end{equation}
			Then $\beta_k >0$ and there exists $u_k \in
			M_\alpha$ such that $\varphi (u_k) = \beta_k$ and $ u_k$ is a weak
			eigenfuction with $\lambda_k = \alpha / \beta_k $.
		\end{proof}

\section{The limit case $p\to\infty$}\label{limitcaseInfty}
    In this section, we study the asymptotic behaviour of
    $\lambda_1(p)^{\nicefrac1p}$ as $p\to\infty.$

    From now on, $u_p$ denotes the positive eigenfunction corresponding to
        $\lambda_1(p)$ such that $\|u_p \|_p=1.$

    In order to pass to the limit in $\mathcal{H}_{J,p}(u,u)$ it is clear that
    for a fixed smooth $u$, it holds that $\|\nabla u\|_p\to \|\nabla u\|_\infty$.
    To deal with the non-local term we use that $J$ is a radially symmetric and compactly supported function to obtain
    $$
    \begin{array}{rl}
    \displaystyle\int_{\mathbb R^N}\int_{\mathbb R^N} J(x-y)|u(x)-u(y)|^pdxdy
    =&\displaystyle \int_{\Omega}\int_{\Omega} J(x-y)|u(x)-u(y)|^pdxdy\\[10pt]
    & \displaystyle + 2 \int_{\Omega}\int_{\Omega_J\setminus\Omega} J(x-y)|u(x)-u(y)|^pdxdy.
    \end{array}
    $$
    where $$\Omega_J=\displaystyle\bigcup_{x\in\Omega} \{x-\mbox{supp}(J)\}.$$ Let us define
    \[
      [u]_J\coloneqq\sup \Big\{|u(x)-u(y)|\colon
            x-y\in \mbox{supp}(J)\Big\}.
    \]
    Thus, for a fixed $u$ we have
    $$
    \begin{array}{rl}\displaystyle
    \left(\int_{\mathbb R^N}\int_{\mathbb R^N} J(x-y)|u(x)-u(y)|^pdxdy\right)^{1/p}\le &\displaystyle
    \left( 2\int_\Omega\int_{\Omega_J} J(x-y)dxdy\right)^{1/p} [u]_J \\[10pt]
     \to& [u]_J \qquad \mbox{as } p\to\infty.
    \end{array}
    $$
    On the other hand, if we let $\Gamma_\delta=\{x\in\Omega, y\in \Omega_J\,:\, J(x-y)>\delta\}$, then
    $$
    \begin{array}{rl}
    \displaystyle\left(\int_{\mathbb R^N}\int_{\mathbb R^N} J(x-y)|u(x)-u(y)|^pdxdy\right)^{1/p}\ge& \displaystyle\left(\delta\iint_{\Gamma_\delta} |u(x)-u(y)|^pdxdy\right)^{1/p}\\
    \to& \displaystyle \sup_{x,y\in\Gamma_\delta} |u(x)-u(y)| \qquad \mbox{as }p \to\infty\\
    \to&  [u]_J\qquad \mbox{as }\delta\to 0.
    \end{array}
    $$
    Hence, we obtain that
    $$
       \left(\int_{\mathbb R^N}\int_{\mathbb R^N} J(x-y)|u(x)-u(y)|^pdxdy\right)^{1/p}
     \to [u]_J \qquad \mbox{as } p\to\infty.
    $$

    Therefore, for a fixed $u$ we have
    \begin{equation}\label{eq.limiteenergia}
\left(\mathcal{H}_{J,p}(u,u)\right)^{\nicefrac1p} \to \max\Big\{\|\nabla u\|_\infty, [u]_J\Big\}\qquad \mbox{as } p\to\infty.
\end{equation}

Now we are ready to prove Theorem \ref{thm:limiteInfty.intro} that says that
        \[
	        \left[\lambda_1(p)\right]^{\nicefrac1p}\to
            \Lambda\coloneqq\inf\left\{
                \dfrac{\max\left\{\|\nabla u\|_{\infty},
                [u]_J\right\}}{\|u\|_\infty}\colon u\in W^{1,\infty}_0(\Omega)
            \right\},
        \]
and that there is a sequence $\{p_n\}_{n\in\mathbb{N}}$
        such that $p_n\to \infty,$ $u_{p_n}\to u_\infty$ uniformly in
        $\overline{\Omega}$ with
         \[
            \Lambda=\frac{\max\Big\{\|\nabla u_\infty\|_{\infty},
                [u_\infty]_J\Big\}}{\|u_\infty\|_\infty}.
        \]

     \begin{proof}[Proof of Theorem \ref{thm:limiteInfty.intro}]
	    Let $u\in W^{1, \infty}_0(\Omega)$ be a nontrivial function. Then for any $p$ we have that  $u\in W^{1,p}_0(\Omega)\setminus\{0\}$ and therefore
$$
\lambda_1(p) \le \dfrac{\mathcal{H}_{J,p}(u,u)}{\|u\|_p^p}.
$$
Then, by \eqref{eq.limiteenergia},
        \[
            \limsup_{p\to\infty} [\lambda_1(p)]^{\nicefrac1p}
            \le \dfrac{ \max \Big\{\|\nabla u\|_\infty, [u]_J\Big\}}{\|u\|_\infty}.
        \]
        Since $u$ is arbitrary, we get
        \begin{equation}\label{eq:limsup}
	        \limsup_{p\to\infty} [\lambda_1(p)]^{\nicefrac1p}\le\Lambda.
        \end{equation}

        Our next aim is to show that
        \[
            \liminf_{p\to\infty} [\lambda_1(p)]^{\nicefrac1p}\ge\Lambda.
        \]
        Let $\{p_n\}_{m\in\mathbb{N}}$ be such that $p_n\to\infty$ and
        \begin{align*}
	        \liminf_{p\to\infty}& [\lambda_1(p)]^{\nicefrac1p}
            =\lim_{n\to\infty} [\lambda_1(p_n)]^{\nicefrac1p_n}\\
            &=\lim_{n\to\infty} \|\nabla u_{p_n}\|_{p_n}^{p_n}+
          \int_{\mathbb{R}^N}\int_{\mathbb{R}^N}J(x-y)|u_{p_n}(x)-u_{p_n}(y)|^{p_n}
            dxdy.
        \end{align*}
        Observe that, from \eqref{eq:limsup}, we get
        \begin{equation}\label{eq:cota1}
	         \|\nabla u_{p_n}\|_{p_n}
            \le\left(2 \Lambda\right)^{\nicefrac1{p_n}}
        \end{equation}
       for any $n\in\mathbb{N}.$
       Now, given $q>N$ there is $n_0$ such that $q<p_n$  for all $n\ge n_0.$ Thus,
       using  H\"older's inequality  and \eqref{eq:limsup}, we have
       \begin{align*}
	        \|\nabla u_{p_n}\|_q\le \|\nabla u_{p_n}\|_{p_n}
	        |\Omega|^{\nicefrac1q-\nicefrac1{p_n}}
	        \le \left(2 \Lambda\right)^{\nicefrac1{p_n}}
	        |\Omega|^{\nicefrac1q-\nicefrac1{p_n}}
       \end{align*}
       for all $n\ge n_0.$ Therefore $\{u_{p_n}\}_{n\ge n_0}$ is bounded
       in $W^{1,q}_0(\Omega).$ Thus, by  the Sobolev embedding theorem,
       passing to a subsequence, still denoted by $\{p_n\}_{n\in\mathbb{N}},$
       we have that
       \[
            u_{p_n}\to u_\infty
       \]
       weakly in $W^{1,q}_0(\Omega)$ and uniformly in $\overline{\Omega}.$

       On the other hand, using again H\"older's inequality, we get
       \begin{align*}
	      &\mathcal{H}_{J,q}(u_{p_n},u_{p_n})
          \le \|\nabla u_{p_n}\|_{p_n}^q|\Omega|^{1-\nicefrac{q}{p_n}}\\
          &+\left(
          \int_{\mathbb{R}^N}\int_{\mathbb{R}^N}
          J(x-y)|u_{p_n}(x)-u_{p_n}(y)|^{p_n} dx dy\right)^{\nicefrac{q}{p_n}}
          \\
          & \qquad \qquad \times
          \left(
          \int_{\Omega}\int_{\Omega_J}
          J(x-y) dx dy\right)^{1-\nicefrac{q}{p_n}}\\
          &\le\left\{
           |\Omega|^{1-\nicefrac{q}{p_n}}+
           \int_{\Omega}\int_{\Omega_J}
          J(x-y) dx dy\right\}^{1-\nicefrac{q}{p_n}}
          [\lambda_1(p_n)]^{\nicefrac{q}{p_n}}.
       \end{align*}
       Passing to the limit one obtains
       \[
            \left[\mathcal{H}_{J,q}(u_{\infty},
            u_{\infty})\right]^{\nicefrac{1}{q}}
            \le \left\{
           |\Omega|+
           \left(
          \int_{\Omega}\int_{\Omega_J}
          J(x-y) dx dy\right)
          \right\}^{\nicefrac{1}{q}}
          \liminf_{p\to\infty}[\lambda_1(p)]^{\nicefrac{1}{p}}.
       \]
       Observe that the above inequality holds for any $q>N$ (using a diagonal argument),
       and then we get that $u_\infty\in W^{1,\infty}_0(\Omega)$ and taking the limit as
        $q\to\infty$ in the last inequality we obtain
        \[
           \max \Big\{\|\nabla u_\infty\|_{\infty},[u_\infty]_J \Big\}
            \le \liminf_{p\to\infty}[\lambda_1(p)]^{\nicefrac{1}{p}}.
        \]
        Moreover from the uniform convergence and the normalization condition, we
        get $\|u_\infty\|_\infty=1$ and therefore
        \begin{equation}\label{eq:liminf}
              \Lambda\le\max \Big\{\|\nabla u_\infty\|_{\infty},[u_\infty]_J \Big\}
            \le \liminf_{p\to\infty}[\lambda_1(p)]^{\nicefrac{1}{p}}.
        \end{equation}
        Then by \eqref{eq:limsup} and \eqref{eq:liminf} we have that
        \[
            \Lambda= \lim_{p\to\infty}[\lambda_1(p)]^{\nicefrac1p}=
            \max \Big\{\|\nabla u_\infty\|_{\infty},[u_\infty]_J\Big\}.
        \]
     \end{proof}

     Now we will give a characterization of $\Lambda$ in terms of $R_J$ and the inradius
     (the radius of the largest ball contained in $\Omega$). Recall that we introduced the notation
     \[
	    R_\Omega\coloneqq\max \Big\{\mbox{dist}(x,\partial\Omega)\colon x\in\overline{\Omega}\Big\},
     \]
 	for the inradius and $K_\Omega\in\mathbb N\cup \{0\}$ and $b\in[0,R_J)$ as
     \[
        R_\Omega=K_\Omega R_J+b,
     \]
that is, we define $K_\Omega$ as the quotient and $b$ as the remainder of the division between $R_\Omega$ and $R_J$.

\begin{lemma}\label{lema.estimates}
For $b=0$, the parameter $\Lambda$ satisfies the lower estimate
$$
\Lambda\ge \max\left\{\frac{1}{K_\Omega},\frac1{R_\Omega}\right\},
$$
while for $b\ne 0$ we have the following estimate
$$
\Lambda\ge \max\left\{ \frac1{R_\Omega},\frac{1}{K_\Omega+b},\frac{1}{K_\Omega+1}\right\}.
$$
\end{lemma}
\begin{proof}
	    By \cite{JL,JLM}, we have
	    \begin{equation}\label{eq:cl11}
	        \dfrac1{R_\Omega}\le \Lambda.
	    \end{equation}
On the other hand, for $b=0$ it is easy to check directly that if
	    $u\in W^{1,\infty}_0(\Omega)$ then
$$
\|u\|_\infty\le  K_\Omega [u]_J \qquad R_J<R_\Omega
$$
Indeed, for all $x\in\Omega$ there exists a sequence of points $\{x_i\}_{i=0}^{K_\Omega}\in\overline\Omega$
with $|x_{i+1}-x_i|\le R_J$ such that $x_0=x$ and $x_{K_\Omega}\in \partial\Omega$, then
$$
|u(x)|\le |u(x)-u(x_1)|+|u(x_1)-u(x_2)|+\cdots+|u(x_{K_\Omega-1})-u(x_{K_\Omega})| \le K_\Omega [u]_J.
$$
Therefore,
\begin{equation}\label{eq:cl12}
\frac1{K_\Omega}\le \Lambda.
\end{equation}

Notice that  by the same argument if $b\ne 0$,
$$
\|u\|_\infty\le (K_\Omega+1)[u]_J.
$$
which implies
\begin{equation}\label{eq:cl13}
\frac1{K_\Omega+1}\le \Lambda.
\end{equation}
In order to obtain the second estimate for $b\ne 0$ we note that
    \begin{itemize}
        \item If $\mbox{dist}( x,\partial \Omega)\le R_J K_\Omega$ then
    	    \[
    	        |u(x)|\le K_\Omega[u]_J\le K_\Omega\max\Big\{[u]_J,\|\nabla u\|_\infty\Big\};
            \]
        \item If
    	    $\mbox{dist}( x,\partial \Omega)> R_J K_\Omega$
	        then
	        \begin{align*}
	            |u(x)|&\le K_\Omega[u]_J+(\mbox{dist}( x,\partial \Omega)-K_\Omega R_J)
	                        \|\nabla u\|_\infty\\
	              &\le K_\Omega[u]_J+b
	                        \|\nabla u\|_\infty\\
	              &\le (K_\Omega+b)\max\Big\{[u]_J,\|\nabla u\|_\infty\Big\}.
            \end{align*}
    \end{itemize}
Thus,
\begin{equation}\label{eq:cl14}
\frac1{K_\Omega+b}\le \Lambda.
\end{equation}
Finally, by \eqref{eq:cl11}, \eqref{eq:cl12}, \eqref{eq:cl13}, and \eqref{eq:cl14} the result follows.
\end{proof}

To characterize $\Lambda$ we consider four different cases.

\begin{theorem}
\label{thm:pco}
	    If $b=0$ then
	    \[
	        \Lambda=\max\left\{\dfrac1{R_\Omega},\dfrac{1}{K_\Omega}\right\}.
	    \]
	    While, if $K_\Omega=0$ or $K_\Omega\ne 0$ with $b\ne0$ and $R_J\le 1$ then
	    \[
	        \Lambda=\max\left\{\dfrac1{R_\Omega},\dfrac{1}{K_\Omega+1}\right\}.
	    \]
\end{theorem}
\begin{proof}
Let $x_0\in \Omega$ such that
        $R_\Omega=\mbox{dist}(x_0,\partial\Omega)$ and define
        \[
            v(x)=\dfrac{(R_\Omega-|x-x_0|)_+}{R_\Omega.}
        \]

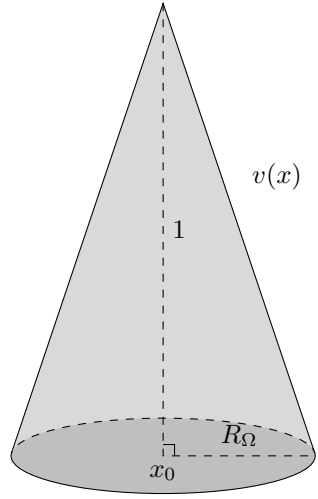
        \begin{center}
        	\begin{tikzpicture}

				\newcommand{\radiusx}{2}
				\newcommand{\radiusy}{.5}
				\newcommand{\height}{6}

				\coordinate (a) at
				(-{\radiusx*sqrt(1-(\radiusy/\height)*(\radiusy/\height))},
				{\radiusy*(\radiusy/\height)});

				\coordinate (b) at
				({\radiusx*sqrt(1-(\radiusy/\height)*(\radiusy/\height))},
				{\radiusy*(\radiusy/\height)});

				\draw[fill=gray!30] (a)--(0,\height)--(b)--cycle;

				\fill[gray!50] circle (\radiusx{} and \radiusy);

				\begin{scope}
					\clip ([xshift=-2mm]a) rectangle ($(b)+(1mm,-2*\radiusy)$);
					\draw circle (\radiusx{} and \radiusy);
					\end{scope}

				\begin{scope}
					\clip ([xshift=-2mm]a) rectangle ($(b)+(1mm,2*\radiusy)$);
					\draw[dashed] circle (\radiusx{} and \radiusy);
				\end{scope}

				\draw[dashed] (0,\height)|-(\radiusx,0) node[right, pos=.25]{$1$}
				node[above,pos=.75]{$R_\Omega$};

				\draw (0,.15)-|(.15,0);
				\draw (0,0) node[below] {$x_0$};
				\draw (1.5,4) node[below] {$v(x)$};
			\end{tikzpicture}
		\end{center}
        Observe that $v\in W^{1,\infty}_0(\Omega)$ and
        \[
            \|v\|_{\infty}  =1, \quad
	            \|\nabla v\|_{\infty} =\dfrac{1}{R_\Omega},
        \]
        \[
	            [v]_J=
	                \begin{cases}
	                    1 &\text{ if } R_J> R_\Omega,\\
	                    \dfrac{R_J}{R_\Omega}&\text{ if } R_J\le R_\Omega.\\
                    \end{cases}
        \]
        Then
        \begin{equation}\label{eq:cl15}
	        \dfrac{\max\left\{\|\nabla v\|_\infty,[v]_J\right\}}{\|v\|_\infty}=
	                \begin{cases}
	                     \max\left\{\dfrac{1}{R_\Omega},1\right\}
	                     &\text{ if } R_J> R_\Omega,\\[10pt]
	                    \max\left\{\dfrac{1}{R_\Omega},
	                    \dfrac{R_J}{R_\Omega}\right\} &\text{ if } R_J\le R_\Omega.\\
                    \end{cases}
        \end{equation}
 Now, we can observe that
        \begin{itemize}
	        \item If $K_\Omega =0$ then $\dfrac1{K_\Omega+1}=1;$
	        \item If $b=0$ then $\dfrac{1}{K_\Omega}
	                =\dfrac{R_J}{R_\Omega};$
	        \item If $K_\Omega>0$, $b>0$ and $R_J\le 1$ then
	            \[
	                \dfrac{1}{K_\Omega+1}=\dfrac{R_J}{(K_\Omega+1)R_J}
	                \le\dfrac{R_J}{R_\Omega}\le\dfrac{1}{R_\Omega}.
	            \]
        \end{itemize}
        Then, by Lemma \ref{lema.estimates} and \eqref{eq:cl15}, the theorem follows.
     \end{proof}

To give a complete characterization of $\Lambda$,
     we have to study the case $K_\Omega>0$, $b>0$ and $R_J>1$.
     In this case, we observe the big difference between our nonlocal operator and the
     $p-$Laplacian and the fractional $p-$Laplacian (see \cite{JL,JLM,LL}) due to
     the fact that $\Lambda$ is not achieved by the cone.

   	\begin{theorem}\label{thm:sco}
		 Let $K_\Omega>0$, $b>0$ and $R_J>1$.  If $b\ge 1$ then
		 \[
		 	\Lambda=\frac1{K_\Omega+1}.
		 \]	
	\end{theorem}
	\begin{proof}
		By Lemma \ref{lema.estimates}, we have
		\[
			\Lambda\ge\dfrac1{K_\Omega+1}.
		\]
		Therefore, we only need to show the inverse inequality.
		
		Take again $x_0\in \Omega$ such that
        $R_\Omega=\mbox{dist}(x_0,\partial\Omega)$ and we define:
        \begin{itemize}
			\item For  $i\in\left\{0,\dots,K_\Omega-1\right\}$
				\[
					A_i=[iR_J,iR_J+b], \quad
					B_i=[iR_J+b,(i+1)R_J]
				\]
				\[
					h_i(t)=
						\left(1-\dfrac{i+\chi_{B_i}(t)}
						{K_\Omega+1}\right)-\dfrac{1}{K_\Omega+1}
						\dfrac{t-iR_J}{b} \chi_{A_i}(t);
				\]
    		\item For $i=K_\Omega$
    			\[
						I= [R_\Omega-b,R_\Omega),\quad   			
    				h_{K_\Omega}(t)=\dfrac{1}{K_\Omega+1}\left(
    				\dfrac{R_\Omega-t}{R-K_\Omega R_J}\right)\chi_I(t);
    			\]
    		\item Finally, we take
    			\[
    				h(t)=\sum_{i=0}^{K_\Omega} h_i(t)\quad\text{ and } \quad w(x)=h(|x-x_0|) .
    			\]
		\end{itemize}

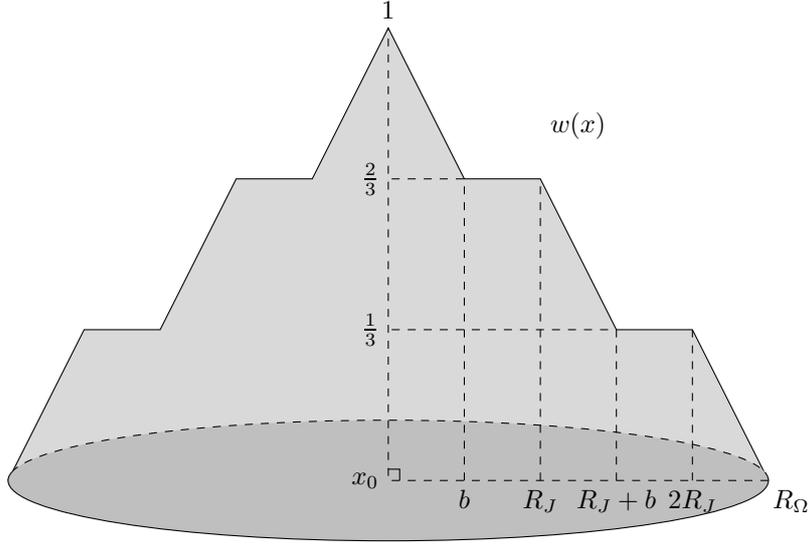
\begin{center}
\begin{tikzpicture}

				\newcommand{\radiusx}{5}
				\newcommand{\radiusy}{.8}
				\newcommand{\height}{6}

				\coordinate (a) at
				(-{\radiusx*sqrt(1-(\radiusy/\height)*(\radiusy/\height))},
				{\radiusy*(\radiusy/\height)});

				\coordinate (b) at
				({\radiusx*sqrt(1-(\radiusy/\height)*(\radiusy/\height))},
				{\radiusy*(\radiusy/\height)});

				
				\path[fill=gray!30](-5,0)--(-4,2)--(-3,2)--(-2,4)--(-1,4)--(0,6)
				--(1,4)--(2,4)--(3,2)--(4,2)--(5,0);
				\draw (-5,0)--(-4,2)--(-3,2)--(-2,4)--(-1,4)--(0,6)
				--(1,4)--(2,4)--(3,2)--(4,2)--(5,0);
				
				\fill[gray!50] circle (\radiusx{} and \radiusy);

				\begin{scope}
					\clip ([xshift=-2mm]a) rectangle ($(b)+(1mm,-2*\radiusy)$);
					\draw circle (\radiusx{} and \radiusy);
					\end{scope}

				\begin{scope}
					\clip ([xshift=-2mm]a) rectangle ($(b)+(1mm,2*\radiusy)$);
					\draw[dashed] circle (\radiusx{} and \radiusy);
				\end{scope}

				\draw (0,.15)-|(.15,0);
				\draw (0,0) node[left] {$x_0$};
				\draw [dashed] (5,0) -- (0,0) -- (0,6) node[above] {$1$};
				\draw (2.5,5) node[below] {$w(x)$};
				\draw (5.3,0) node[below] {$R_\Omega$};
				\draw [dashed] (1,0) node[below] {$b$} -- (1,4) -- (0,4)
				node[left] {$\frac{2}3$};
				\draw [dashed] (2,0) node[below] {$R_J$} -- (2,4);
				\draw [dashed] (3,0) node[below] {$R_J+b$} -- (3,2) -- (0,2)
				 node[left] {$\frac{1}3$};
				 \draw [dashed] (4,0) node[below] {$2R_J$} -- (4,2);
			 \end{tikzpicture}
			
			 The function $w$ when $K_\Omega=2.$
		\end{center}

		Observe that $w\in W^{1,\infty}_0(\Omega),$ $\|w\|_\infty=1,$
		\[
			\|\nabla w\|_\infty=\dfrac{1}{K_\Omega+1}\dfrac1{b}
			\quad \text{ and }\quad
			[w]_J=\dfrac{1}{K_\Omega+1}.
		\]
		
		Then, if $b\ge1,$ we have
		\[
			\Lambda\le\max\Big\{ \|\nabla w\|_\infty,[w]_J \Big\}=\dfrac{1}{K_\Omega+1}.
		\]
	\end{proof}

	\begin{theorem}\label{thm:tco}
		 Let $K_\Omega>0$, $b>0$ and $R_J>1$. If $b<1$ then
		 \[
		 	\Lambda=\dfrac1{K_\Omega+b}.
		 \]	
	\end{theorem}
	\begin{proof}
Let us observe that as $b<1<R_J$ we get
$$
R_\Omega+b<R_\Omega+1,\qquad K_\Omega+b<K_\Omega R_J+b=R_\Omega
$$
then by Lemma \ref{lema.estimates}
$$
\Lambda\ge \dfrac1{K_\Omega+b}
$$
Therefore, we only need to show the inverse inequality.
		
		Take again $x_0\in \Omega$ such that
        $R_\Omega=\mbox{dist}(x_0,\partial\Omega)$ and we define:
						


       \begin{itemize}
			\item For  $i\in\left\{0,\dots,K_\Omega-1\right\}$
				\[
					A_i=[iR_J,iR_J+b], \quad
					B_i=[iR_J+b,(i+1)R_J]\quad
                    y_i=1-\frac{i}{K_\Omega+b}
				\]
				\[
					f(t)=
                        \left\{\begin{array}{ll}
                            y_i-\dfrac{1}{K_\Omega+b} (t-i R_J)\qquad & t\in A_i,\\ \\
                            y_i-\dfrac{b}{K_\Omega+b}-\dfrac{1-b}{(K_\Omega+b)(R_J-b)}
                            (t-i R_J-b)& t\in B_i;
                        \end{array}\right.						
				\]
    		\item For $i=K_\Omega$
    			\[ 			
    				f(t)=\dfrac{b}{K_\Omega+b}-\frac{1}{K_\Omega+b}
    				(t- R_\Omega+b)\quad t\in [R_\Omega-b,R_\Omega);
    			\]
             \item $z\colon\mathbb{R}^N\to\mathbb{R}$
			\[
				z(x)= \begin{cases}
						f(|x-x_0|) &\text{ if } |x-x_0|\le R_\Omega,\\
						0 &\text{ if } |x-x_0|> R_\Omega.\\
						\end{cases}
			\]
		\end{itemize}
		
		\begin{center}
	   \begin{tikzpicture}[scale=3]
				\draw (0,1)--(0.6,2/2.6)--(1.2,1.6/2.6)--(1.8,1/2.6) -- (2.4,.6/2.6) -- (3,0);
				\draw (0,1) node[left] {$1$}  -- (0,0)  node[below] {$0$} --
				(0.6,0) node[below] {$b$}
				-- (1.2,0)  node[below] {$R_J$} -- (1.8,0)
				 node[below] {$R_J+b$} -- (2.4,0) node[below] {$2R_J$}-- (3,0)
				 node[below] {$R_\Omega$};
				 \draw [dashed] (.6,0) -- (0.6,2/2.6) --(0,2/2.6)
				 node[left] {$\frac2{2+b}$};
				 \draw [dashed] (1.2,0) -- (1.2,1.6/2.6) --(0,1.6/2.6)
				 node[left] {$\frac{1+b}{2+b}$};
				 \draw [dashed] (1.8,0) -- (1.8,1/2.6) --(0,1/2.6)
				 node[left] {$\frac{1}{2+b}$};
				 \draw [dashed] (2.4,0) -- (2.4,0.6/2.6) --(0,.6/2.6)
				 node[left] {$\frac{b}{2+b}$};
		\end{tikzpicture}
		
		The function $f$ when $K_\Omega=2.$
		\end{center}
		Observe that $z\in W^{1,\infty}_0(\Omega),$ $\|z\|_\infty=1,$
		 \[
		 	[z]_J=\dfrac{1}{K_\Omega+b}
		 \]
		and
		\[
			\|\nabla z\|_\infty=\max\left\{
			\dfrac1{K_\Omega+b},
			\dfrac{1-b}{(K_\Omega+b)
			(R_J-b)}
			\right\}=\frac{1}{K_\Omega+b}
		\]
        because $R_J>1$. Thus,
		\[
			\Lambda\le \dfrac1{K_\Omega+b}
		\]
		and so finish the proof.
	\end{proof}

\section{The limit problem}
    Our  last  aim  is  show  that $u_\infty$
    is  a  viscosity  solution  of
    \begin{equation}
      \label{eq:limite}
            \begin{cases}
	            \max\Big\{M_1(u), M_2(u)\Big\}=0 &\text{ in }\Omega,\\
	            u =0 &\text{ on }\partial\Omega,
            \end{cases}
    \end{equation}
    where
    $$
    	    M_1(u(x))\coloneqq\min \Big\{-\Lambda u(x)-\mathcal{L}^-_{J,\infty}u(x),
	        -\mathcal{L}_{J,\infty}u(x), -\mathcal L_{\infty,J}^- u(x)-|\nabla u(x)|
	        \Big\}
	        $$
	        and
	           \begin{align*}
            & M_2(u(x))\coloneqq
             \min\Big \{|\nabla u(x)|-\Lambda u(x),
                -\Delta_\infty u(x),
                |\nabla u(x)|-\mathcal L_{\infty,J}^+u(x),\\
                &\qquad\qquad \qquad \qquad\qquad |\nabla u (x)|+\mathcal L_{\infty,J}^-u(x)
                \Big\}.
     \end{align*}

The following result follows from \cite[Lemma 6.5]{CHM}. We include the proof for reader's convenience.
\begin{lemma}\label{lem_CHM}
  Let $\phi\in C_0^1(\overline\Omega)$ extended by zero outside $\Omega$ and $x_p\to x_0$ as $p\to\infty$. Then,
  $$
  \left(\int_{\mathbb{R}^N}J(x_p-y)(\phi(y)-\phi(x_p))_+^{p-1}\, dy \right)^{\frac1{p-1}}\to \mathcal L^+_{J,\infty}\phi(x_0)
  $$
  and
  $$
  \left(\int_{\mathbb{R}^N}J(x_p-y)(\phi(y)-\phi(x_p))_-^{p-1}\, dy \right)^{\frac1{p-1}} \to -\mathcal L^-_{J,\infty}\phi(x_0).
  $$
\end{lemma}
\begin{proof}
We just examine the first convergence, since the same argument can be used to analyze the second one. Let us define
$$
A_p=\int_{\mathbb{R}^N}J(x_p-y)(\phi(y)-\phi(x_p))_+^{p-1}\, dy \quad\mbox{and}\quad
\mu=\sup_{x_0-y\in supp(J)} \Big(\phi(y)-\phi(x_0)\Big).
$$

The upper bound follows easily
$$
A_p^{\frac1{p-1}} \le \sup_{x_p-y\in supp(J)} \Big(\phi(y)-\phi(x_p)\Big)_+ \to\mu \qquad \mbox{as }p\to\infty
$$
To obtain the lower estimate we assume that $\mu>0$. Observe that as $A_p\ge0$, if $\mu=0$ we are done. Taking $0<t<\mu$ there exists $y_0\in B_{R_J}(x_0)$ such that $\phi(y_0)-\phi(x_0)>t$. Since $\phi$ is continuous there exist $\delta>0$ such that if $x\in B_{\delta}(x_0)$ and $y\in B_{\delta}(y_0)\subset B_{R_J}(x_0)$ then $\phi(y)-\phi(x)>t$. Therefore
$$
A_p^{\frac1{p-1}}\ge t \left(\int_{B_{\delta}(y_0)} J(x_p-y)\,dy\right)^{\frac{1}{p-1}} \to t \quad \mbox{as } p\to\infty.
$$
Finally, taking $t\to\mu$, we obtain the desired result.
\end{proof}
    \begin{proof}[Proof of Theorem \ref{teo.viscosa.intro}]
        We begin the proof, showing that $u_\infty$ is a viscosity supersolution of
        \eqref{eq:limite}.
	
	    Fix a open set $U\subset\Omega,$  a point $x_0\in U$ and a function
		$\psi\in C^2(U)$ such that $\psi(x_0)=u(x_0),$ $\psi< u$ in $U\setminus\{x_0\}.$
		
		Proceeding as in \cite{JLM}, we can show that there exist
		$\{p_n\}_{n\in\mathbb{N}}$ and $\{x_n\}_{n\in\mathbb{N}}\subset \Omega$ such that
		\[
		    p_n\to\infty, \quad x_n\to x_0\quad\text{ as } n\to\infty,
		\]
		\[
		    u_{p_n}\to u_\infty\quad \text{uniformly in } \overline{\Omega}
		    \quad\text{ as } n\to\infty,
		\]
		and $u_{p_n}-\psi$ attains its minimum at point $x_n$ for all $n\in\mathbb{N}.$
		
		Since $u_{p_n}$ is a viscosity solution of \eqref{eq:EP} with
		$\lambda=\lambda_1(p),$ we get
		\[
		    \mathcal{L}_{J,p_n} w_{n}(x_n)\ge
		    \lambda_1(p_n)\left(u_{p_n}(x_n)\right)^{p_n-1}\quad \forall
		    n\in\mathbb{N}
		\]
		where
		\[
						w_n(x)=
							\begin{cases}
								\psi(x) & \text{if } x\in U,\\
								u_{p_n}(x) & \text{if } x\in U^c,
							\end{cases}
	    \]	
	    for any $n\in\mathbb{N}.$
	    Thus
	    \begin{align*}
		    &\lambda_1(p_n)\left(u_{p_n}\right)^{p_n-1}
		    \le \mathcal{L}_{J,p_n} w_n(x_n)\\
		    &=-\left[
		        |\nabla w_n(x_n)|^{p_n-2}\Delta w_n(x_n)+(p_n-2) |\nabla w_n(x_n)|^{p_n-4}
		        \Delta_\infty w_n(x_n) \right] \\
		    &\quad-2\int_{\mathbb{R}^N}J(x_n-y)(w_n(y)-w_n(x_n))_+^{p_n-1}\, dy\\
		    &\quad+2\int_{\mathbb{R}^N}J(x_n-y)(w_n(y)-w_n(x_n))_-^{p_n-1}\, dy.
        \end{align*}
        Then
        \begin{equation}\label{eq:vs1}
	             \begin{aligned}
		                &\lambda_1(p_n)\left(u_{p_n}(x_n)\right)^{p_n-1}
		                +2\int_{\mathbb{R}^N}J(x_n-y)(w_n(y)-w_n(x_n))_+^{p_n-1}\, dy
		                \\
		                 &\le -\left[
		               |\nabla w_n(x_n)|^{p_n-2}\Delta w_n(x_n)+(p_n-2)
		               |\nabla w_n(x_n)|^{p_n-4}
		            \Delta_\infty w_{n}(x_n) \right] \\
		            &\qquad +2\int_{\mathbb{R}^N}J(x_n-y)(w_n(y)-w_n(x_n))_-^{p_n-1}\, dy.
                \end{aligned}
        \end{equation}

        Suppose that there is a subsequence $\{x_{n_j}\}_{j\in\mathbb{N}}$ such that
        $x_j\to x_0$ as $j\to \infty$ and
        \[
            \begin{aligned}
	             &\int_{\mathbb{R}^N}J(x_{n_j}-y)(w_{n_j}(y)-w_{n_j}(x_{n_j}))_-^{p_{n_j}-1}\, dy
	            \\
                  &   \ge - \left[
		               |\nabla w_{n_j}(x_{n_j})|^{p_{n_j}-2}\Delta
		               w_{n_j}(x_{n_j})+(p_{n_j}-2)
		               |\nabla w_{n_j}(x_{n_j})|^{p_{n_j}-4}
		            \Delta_\infty w_{n_j}(x_{n_j}) \right].
            \end{aligned}
       \]
         Then, using \eqref{eq:vs1}, we get
         \begin{align*}
	             &\left(
	             \int_{\mathbb{R}^N}J(x_{n_j}-y)(w_{n_j}(y)-w_{n_j}(x_{n_j}))_-^{p_{n_j}-1}\, dy
	             \right)^{\frac1{p_{n_j-1}}}
	             \\
                  & \ge \left((p_{n_j}-2)
                  |\nabla w_{n_j}(x_{n_j})|^{p_{n_j}-4}\right)^{\frac{1}{p_{n_j}-1}}\left|
		               \dfrac{|\nabla w_{n_j}(x_{n_j})|^{2}}{p_{n_j}-2}
		               \Delta w_{n_j}(x_{n_j})+
		            \Delta_\infty w_{n_j}(x_{n_j}) \right|^{\frac{1}{p_{n_j}-1}},
            \end{align*}
           and 
        \begin{align*}
	            &\Bigg(\lambda_1(p_{n_j})
	            \left(u_{p_{n_j}} (x_{n_j})\right)^{p_{n_j}-1}\\
		           &\qquad +2
		           \int_{\mathbb{R}^N}J(x_{n_j}-y)(w_{n_j}(y)-
		           w_{n_j} (x_{n_j})(x_{n_j}))_+^{p_{n_j}-1}\, dy
		           \Bigg)^{\frac1{p_{n_j}-1}}\\
		            &\le
		           \left(3 \int_{\mathbb{R}^N}J(x_{n_j}-y)(w_{n_j}(y)-
		           w_{n_j} (x_{n_j})(x_{n_j}))_-^{p_{n_j}-1}
		           \, dy  \right)^{\frac1{p_{n_j}-1}}
        \end{align*}
         for any $j\in\mathbb{N}.$
         Passing to the limit as $n_j\to \infty$ by Lemma \ref{lem_CHM}, we get
         \begin{align*}
            &|\nabla w(x_0)|\le -\mathcal{L}^-_{J,\infty}w(x_0),\\
	        &\max\Big\{\Lambda w(x_0),\mathcal{L}^+_{J,\infty}w(x_0)\Big\}\le
	            -\mathcal{L}^-_{J,\infty}w(x_0),
        \end{align*}
         where
         \[
						w(x)=
							\begin{cases}
								\psi(x) & \text{if } x\in U,\\
								u_{\infty}(x) & \text{if } x\in U^c.
							\end{cases}
	    \]	
         Therefore,
         \[
	        \min\Big\{-\Lambda w(x_0)-\mathcal{L}^-_{J,\infty}w(x_0),
	        -\mathcal{L}_{J,\infty}w(x_0),
            -\mathcal L_{\infty,J}^- w(x_0)-|\nabla w(x_0)|\Big\}\ge 0,
         \]
         that is
         \begin{equation}\label{eq:vs2}
            M_1(w(x_0))\ge 0.
         \end{equation}

         On the other hand, if the subsequence $\{x_{n_j}\}_{j\in\mathbb{N}}$
         does not exist, then there is $n_0$  such that
         \[
            \begin{aligned}
	             &\int_{\mathbb{R}^N}J(x_n-y)(w_{n}(y)-w_{n}(x_{n}))_-^{p_{n}-1}\, dy
	             \\
                  & \le -\left[
		               |\nabla w_{n}(x_{n})|^{p_{n}-2}\Delta
		               w_{n_j}(x_{n})+(p_{n}-2)
		               |\nabla w_{n}(x_{n})|^{p_n-4}
		            \Delta_\infty w_{n}(x_{n}) \right]
            \end{aligned}
       \]
        for all $n\ge n_0$.
        Then, from \eqref{eq:vs1} we get that $\nabla w_{n}(x_{n})\neq 0,$
        \begin{align*}
	             &\dfrac{|\nabla w_n(x_n)|^3}{p_n-2}\left(\dfrac{\left(
	             \int_{\mathbb{R}^N}J(x_n-y)(w_{n}(y)-w_{n}(x_{n}))_-^{p_{n}-1}\, dy
	             \right)^{\frac{1}{p_{n-1}}}}{|\nabla w_n(x_n)|}\right)^{p_n-1}
	             \\
                  &\le - \dfrac{|\nabla w_{n}(x_{n})|^{2}}{p_{n}-2}
		               \Delta w_{n}(x_{n})
		            - \Delta_\infty w_{n}(x_{n}) ,
        \end{align*}
        and
         \begin{equation}\label{eq:vs3}
        \begin{aligned}
	            &
	            \dfrac{\lambda_1(p_{n})^{\frac{4}{p_{n}}}
	            (u_{p_n}(x_{n}))^3}{3(p_{n}-2)}
	            \left(
	     \dfrac{\lambda_1(p_{n})^{\frac{1}{p_{n}}}u_{p_n}(x_{n})}
	     {|\nabla w_n(x_{n})|}
	            \right)^{p_{n}-4}\\
		           &\qquad +\dfrac{2|\nabla w_{n}(x_n)|^3}{3(p_{n}-2)}\left(\dfrac{\left(
		           \int_{\mathbb{R}^N}J(x_n-y)
		           (w_{n}(y)-w_n(x_{n}))_+^{p_{n}-1}\, dy
		           \right)^{\frac1{p_n-1}}}{|\nabla w_n(x_n)|}\right)^{p_n-1}\\
		            &\le-3\left[
		        \dfrac{|\nabla w_n(x_{n})|^{2}}{p_{n}-2}
		        \Delta w_n(x_{n})+\Delta_\infty w_n(x_{n}) \right].
        \end{aligned}
       \end{equation}
        Since in the previous inequalities the right hand side is bounded, we get
        \begin{equation}\label{eq:vs4}
	        \begin{aligned}
	            \dfrac{-\mathcal L_{\infty,J}^-w(x_0)}{|\nabla w(x_0)|}&\le1,\\
	            \dfrac{\Lambda w(x_{0})}
	            {|\nabla w(x_{_0})|}&\le 1\\
	             \dfrac{-\mathcal L_{\infty,J}^+w(x_0)}{|\nabla w(x_0)|}&\le1,\\
            \end{aligned}
        \end{equation}

        Therefore, passing to the limit in \eqref{eq:vs3},
        using again Lemma \ref{lem_CHM} we obtain
        \[
	         \begin{array}{rl}
            &\min\Big\{|\nabla w(x_0)|-\Lambda w(x_0),
                -\Delta_\infty w(x_0),
                |\nabla w(x_0)|-\mathcal L_{\infty,J}^+w(x_0),\\[10pt]
                &\qquad \qquad |\nabla w (x_0)|+\mathcal L_{\infty,J}^-w(x_0)
                \Big\}\ge0
            \end{array}
        \]
        that is
        \begin{equation}\label{eq:vs5}
	        M_2(w(x_0))\ge0.
        \end{equation}
        Therefore, by \eqref{eq:vs2} and \eqref{eq:vs5}
        \[
            \max\Big\{M_1(w(x_0)), M_2(w(x_0))\Big\}\ge0
        \]
        Then $u_\infty$ is a viscosity supersolution of \eqref{eq:limite}.

        To finish the proof, we need to show that $u_\infty$ is a viscosity
        subsolution of \eqref{eq:limite}.
        Fix a open set $U\subset\Omega,$  a point $x_0\in U$ and a function
		$\psi\in C^2(U)$ such that $\phi(x_0)=u(x_0),$ $\phi> u$ in $U\setminus\{x_0\}.$
        We want to show that
         \[
            \max\Big\{M_1(v(x_0)), M_2(v(x_0))\Big\}\le0
        \]
         where
         \[
						v(x)=
							\begin{cases}
								\phi(x) & \text{if } x\in U,\\
								u_{\infty}(x) & \text{if } x\in U^c.
							\end{cases}
	    \]	
	
	    Proceeding as in \cite{JLM}, we can show that there exist
		$\{p_n\}_{n\in\mathbb{N}}$ and $\{x_n\}_{n\in\mathbb{N}}\subset \Omega$ such that
		\[
		    p_n\to\infty, \quad x_n\to x_0\quad\text{ as } n\to\infty,
		\]
		\[
		    u_{p_n}\to u_\infty\quad \text{uniformly in } \overline{\Omega}
		    \quad\text{ as } n\to\infty,
		\]
		and $u_{p_n}-\phi$ attains its maximum at point $x_n$ for all $n\in\mathbb{N}.$
		
		Since $u_{p_n}$ is a viscosity solution of \eqref{eq:EP} with
		$\lambda=\lambda_1(p),$ we get
		\begin{equation}\label{eq:vs6}
	        \mathcal{L}_{J,p_n} v_{n}(x_n)\le
		    \lambda_1(p_n)\left(u_{p_n}(x_n)\right)^{p_n-1}\quad \forall
		    n\in\mathbb{N}
        \end{equation}
		where
		\[
						v_n(x)=
							\begin{cases}
								\phi(x) & \text{if } x\in U,\\
								u_{p_n}(x) & \text{if } x\in U^c,
							\end{cases}
	    \]	
	    for any $n\in\mathbb{N}.$ Then
	             \begin{align*}
		                &\lambda_1(p_n)\left(u_{p_n}(x_n)\right)^{p_n-1}
		                +2\int_{\mathbb{R}^N}J(x_n-y)(v_n(y)-v_n(x_n))_+^{p_n-1}\, dy
		                \\
		                 &\ge -\left[
		               |\nabla v_n(x_n)|^{p_n-2}\Delta v_n(x_n)+(p_n-2)
		               |\nabla v_n(x_n)|^{p_n-4}
		            \Delta_\infty v_{n}(x_n) \right] \\
		            &\qquad +2\int_{\mathbb{R}^N}J(x_n-y)(v_n(y)-v_n(x_n))_-^{p_n-1}\, dy,
                \end{align*}
            and therefore
	        \begin{align*}
		                &\Bigg(\lambda_1(p_n)\left(v_{p_n}(x_n)\right)^{p_n-1}
		                +2\int_{\mathbb{R}^N}J(x_n-y)(v_n(y)-v_n(x_n))_+^{p_n-1}\, dy
		                \\
		                 &\qquad + (p_n-2)
		               |\nabla v_n(x_n)|^{p_n-4}\left|
		               \dfrac{|\nabla v_n(x_n)|^{2}}{p_n-2}\Delta v_n(x_n)+
		            \Delta_\infty v_{n}(x_n) \right|\Bigg)^{\frac{1}{p_n-1}} \\
		            &\ge
		            \left(2\int_{\mathbb{R}^N}J(x_n-y)(v_n(y)-v_n(x_n))_-^{p_n-1}\, dy
		            \right)^{\frac1{p_n-1}}.
          \end{align*}
          Thus, passing to the limit 
            using again Lemma \ref{lem_CHM} we obtain
            \[
                \max\Big\{\Lambda v(x_0),\mathcal L_{\infty,J}^+v(x_0),|\nabla v(x_0)|\Big\}
                \ge-\mathcal L_{\infty,J}^-v(x_0).
            \]
           Hence
            \[
                0
                \ge\min\Big\{-\Lambda v(x_0)-\mathcal L_{\infty,J}^-v(x_0),
                -\mathcal L_{\infty,J} v(x_0),-|\nabla v(x_0)|
                -\mathcal L_{\infty,J}^-v(x_0)\Big\},
            \]
            that is
            \begin{equation}\label{eq:vs7}
	            M_1(v(x_0))\le0.
            \end{equation}

          On the other hand, if $|\nabla v(x_0)|=0$ then
          \begin{equation}\label{eq:vs8}
	            M_2(v(x_0))\le0.
          \end{equation}
          We now assume that  $|\nabla v(x_0)|>0$ and
          \begin{equation}\label{eq:vs9}
	            \min\Big\{|\nabla v(x_0)|-\Lambda v(x_0),
                |\nabla v(x_0)|-\mathcal L_{\infty,J}^+v(x_0),\\
                |\nabla v (x_0)|+\mathcal L_{\infty,J}^-v(x_0)
                \Big\}>0.
           \end{equation}

          By \eqref{eq:vs6}, we get
          \begin{align*}
	            &
	            \dfrac{\lambda_1(p_{n})^{\frac{4}{p_{n}}}
	            (u_{p_n}(x_{n}))^3}{3(p_{n}-2)}
	            \left(
	     \dfrac{\lambda_1(p_{n})^{\frac{1}{p_{n}}}u_{p_n}(x_{n})}
	     {|\nabla v_n(x_{n})|}
	            \right)^{p_{n}-4}\\
		           &\qquad +\dfrac{2|\nabla v_{n}(x_n)|^3}{3(p_{n}-2)}\left(\dfrac{\left(
		           \int_{\mathbb{R}^N}J(x_n-y)
		           (v_{n}(y)-v_n(x_{n}))_+^{p_{n}-1}\, dy
		           \right)^{\frac1{p_n-1}}}{|\nabla v_n(x_n)|}\right)^{p_n-1}\\
		            &\ge-\left[
		        \dfrac{|\nabla v_n(x_{n})|^{2}}{p_{n}-2}
		        \Delta v_n(x_{n})+\Delta_\infty v_n(x_{n}) \right].
        \end{align*}
        Thus, passing to the limit and
            using \eqref{eq:vs9} and  Lemma \ref{lem_CHM} we get
            \[
                0\ge-\Delta_\infty v(x_{0}).
            \]
            Therefore
            \begin{equation}\label{eq:vs10}
	            M_2(v(x_0))\le0.
            \end{equation}

            Finally by \eqref{eq:vs7},\eqref{eq:vs8}, and \eqref{eq:vs10}, we have that
            \[
                \max\Big\{M_1(v(x_0)),M_2(v(x_0))\Big\}\le0
            \]
            and therefore $u_\infty$ is a viscosity
        subsolution of \eqref{eq:limite}.
    \end{proof}

    \section*{Acknowledgements}
    LDP and JDR are partially supported by PIP GI No 11220150100036CO. CONICET (Argentina).
    JDR is also supported by MINECO MTM2015-70227-P (Spain) and
Subsidio 20020160100155BA. UBACyT (Argentina). RF is supported by the Spanish project MTM2017-87596-P.

\end{document}